\documentclass{amsart}
\usepackage{amsmath, amssymb,amsthm,latexsym,verbatim}
\usepackage{enumerate}

\makeatletter
\@namedef{subjclassname@2010}{%
  \textup{2010} Mathematics Subject Classification}
\makeatother

\newtheorem{thm}{Theorem}
\newtheorem{cor}[thm]{Corollary}
\newtheorem{prop}[thm]{Proposition}
\newtheorem{lem}[thm]{Lemma}

\theoremstyle{definition}
\newtheorem{defin}[thm]{Definition}
\newtheorem{rem}[thm]{Remark}

\DeclareMathOperator{\AGL}{AGL}

\DeclareMathOperator{\GL}{GL}

\newcommand{\F}{\ensuremath{{\mathbb{F}}}}

\newcommand{\ord}{\textup{ord}\,}

\begin{document}

\title[Affine-Linear Functional Graphs]{On the Number of Distinct Functional Graphs of Affine-Linear Transformations over Finite Fields}

\author[E. Bach and A. Bridy]{Eric Bach and Andrew Bridy}
\address{Eric Bach\\Department of Computer Sciences\\ University of Wisconsin-Madison\\
Madison, WI 53706, USA}
\email{bach@cs.wisc.edu}
\address{Andrew Bridy\\Department of Mathematics\\ University of Wisconsin-Madison\\
Madison, WI 53706, USA}
\email{bridy@math.wisc.edu}

\date{\today}

\begin{abstract}
We study the number of non-isomorphic functional graphs of affine-linear transformations from $(\F_q)^n$ to itself, and we prove upper and lower bounds on this quantity as $n\to\infty$. As a corollary to our result, we prove bounds on the number of conjugacy classes in the symmetric group $S_{q^n}$ that intersect $\AGL_n(q)$. 
\end{abstract}

\subjclass[2010]{Primary 37P05; Secondary 05C20, 11T55}

\keywords{Linear Algebra, Finite Fields, Functional Graphs}

\maketitle
\section{Introduction}

Let $X$ be a finite set and let $f:X\to X$ be a function, so that the pair $(X,f)$ defines a discrete dynamical system. We define the \emph{functional graph} of $(X,f)$ to be a directed graph with vertices at each element of $X$ and an edge from $x$ to $y$ if and only if $f(x)=y$. We denote the functional graph of $(X,f)$ by $G_{(X,f)}$.

In the setting where $X=(\F_q)^n$ and $f$ is a linear transformation, the structure of $G_{(X.f)}$ was explicitly determined by Elspas, who worked in the context of linear sequential networks in electrical engineering \cite{Elspas}. Wang, building on the work of Elspas, provided an explicit description of $G_{(X.f)}$ in the more general case where $f$ is affine-linear \cite{Wang}.

Rather than analyze the structure of the functional graph, our problem is to estimate the number of non-isomorphic functional graphs of affine-linear transformations from $(\F_q)^n$ to itself. Let $D_q(n)$ denote this quantity. We will prove the following:

\begin{thm}\label{thm: main}
For $q$ fixed and $n\to\infty$,
\begin{equation*}
\sqrt{n}\ll \log D_q(n) \ll \frac{n}{\log\log n}.
\end{equation*}
\end{thm}

The upper bound in Theorem \ref{thm: main} is in some sense an improvement on the well-known fact that the number of conjugacy classes in $\GL_n(q)$ is $q^n+O(q^{\lfloor\frac{n-1}{2}\rfloor})$ \cite{Stanley}, as similar linear transformations have isomorphic functional graphs. In fact, Theorem \ref{thm: main} implies that there exist linear transformations which are ``similar" under conjugation by a non-linear permutation of $(\F_q)^n$. We formalize this notion as follows:

\begin{defin}\label{def: conjugate}
The dynamical systems $(X,f)$ and $(Y,g)$ are \emph{dynamically equivalent} if there exists a bijection $\sigma:X\to Y$ such that $\sigma^{-1}\circ g\circ\sigma=f$.
\end{defin}

It is easy to check that dynamical equivalence coincides with isomorphism of functional graphs, and we will use both concepts interchangeably. If $(X,f)$ and $(X,g)$ are dynamically equivalent and $X$ is understood, we say that $f$ and $g$ are dynamically equivalent, and we write $f\sim g$ for short. In this language, Theorem \ref{thm: main} can be restated as an estimate of the number of dynamical equivalence classes of affine-linear transformations $f:(\F_q)^n\to(\F_q)^n$.

Our result also has a group-theoretic interpretation. Let $G$ be the symmetric group of $X=(\F_q)^n$, so that $G\cong S_{q^n}$. Let $H=\AGL_n(q)$ be the group of invertible affine-linear transformations of $X$, so that $H\subseteq G$. Conjugacy classes of $G$ are cycle types, which are special isomorphism classes of functional graphs on $X$ (for those $f:X\to X$ that are bijections), so Theorem \ref{thm: main} yields the following corollary:

\begin{cor}\label{cor: Symmetric Group}
Let $A_q(n)$ be the number of conjugacy classes of $G$ that intersect $H$. Then
\begin{equation*}
A_q(n)\leq \exp\left(O\left(\frac{n}{\log\log n}\right)\right).
\end{equation*}
\end{cor}

We can rephrase Corollary \ref{cor: Symmetric Group} in terms of counting derangements, which are permutations with no fixed points. Let $G$ act on the coset space $G/H$ by left multiplication. The number of conjugacy classes of $G$ that do not contain a derangement in this action is at most $\exp\left(O\left(\frac{n}{\log\log n}\right)\right)$. As the total number of conjugacy classes of $G$ is $P(q^n)$, it follows from the Hardy-Ramanujan asymptotic formula for the partition function \cite{Andrews} that the proportion of conjugacy classes of $G$ that contain a derangement approaches 1 as $n\to\infty$ .

There is significant work of Boston and others on the proportion of elements in a permutation group that act as derangements (e.g. \cite{BostonFixedPointFree},\cite{GuralnickWan}). We know of no other results on the proportion of conjugacy classes that contain a derangement.

\begin{rem}
This paper grew out of work on the number of distinct functional graphs that arise from quadratic polynomials $f(x)=a x^2 + b x + c \in\F_{2^n}[x]$, where $X=\F_{2^n}$. If we fix $\zeta\in\F_{2^n}$ of absolute trace 1, an easy argument using Hilbert's Theorem 90 shows that every such $f$ is conjugate by some $\alpha x+\beta\in\F_{2^n}[x]$ to either $x^2+ b x$ or $x^2+ b x + (b^2 + 1)\zeta$, and moreover no distinct polynomials in these two families are conjugate by a linear polynomial in $\F_{2^n}[x]$. (However, the two families collapse into one under conjugation by polynomials $\alpha x+\beta\in\overline{\F}_{2^n}[x]$. This, and the fact that there are exactly two inequivalent maps for each parameter $b$, are predicted by the cohomological theory of twists of dynamical systems \cite{SilvermanTwists}.) It follows from these ``normal forms" that there at at most $2^{n+1}-1$ distinct functional graphs in the quadratic family. 

Every quadratic polynomial is an affine-linear transformation of $X$ as an $\F_2$-vector space, so the improved upper bound in Theorem \ref{thm: main} applies to the special case of quadratic polynomial maps on $X$, but also to a much broader family of $X$ and $f$.
\end{rem}

\section{Bounds on $D_q(n)$}

The following useful definition is taken from \cite{LN} and will play a crucial role in our counting arguments.

\begin{defin}\label{def:order}
The \emph{order} of $f\in\mathbb{F}_q[x]$ with $f(0)\neq 0$ is the smallest positive $n$ such that $f\mid x^n-1$. We write $\ord f$ for the order of $f$.
\end{defin}

We record for later use a proposition counting orders of degree $k$ irreducible polynomials over $\F_q$.

\begin{prop}\label{prop:order count}
Let $O_q(k)$ be the number of integers that occur as orders of irreducible polynomials over $\F_q$ of degree $k$. Then
\begin{equation*}
O_q(k)=\sum_{d\mid k} \tau(q^d-1)\mu(k/d),
\end{equation*}
where $\tau(n)$ is the number of divisors of $n$. In particular, $O_q(k)\leq \tau(q^k-1)$.
\end{prop}
\begin{proof}
For $f\in\F_q[x]$ irreducible of degree $k$ with $f(0)\neq 0$, it is easy to show that $\ord f$ is the multiplicative order of the element $[x]$ in the field $\mathbb{F}_q[x]/(f)\cong \F_{q^k}$ \cite{LN}. Each nonzero $\alpha\in\F_{q^k}$ has an irreducible minimal polynomial over $\F_q$ of degree dividing $k$, and the order of this polynomial is the multiplicative order of $\alpha$. All divisors of $|\F_{q^k}^\times|=q^k-1$ occur as multiplicative orders of some $\alpha\in\F_{q^k}$, and all irreducible polynomials over $\F_q$ of degree dividing $k$ split in $\F_{q^k}$, so this proves
\begin{equation*}
\sum_{d\mid k} O_q(d) = \tau(q^k-1)
\end{equation*}
and the proposition follows by M\"obius inversion.
\end{proof}

In Proposition \ref{Inequivalent by Partitions} we construct a family of dynamically inequivalent linear transformations of $(\F_q)^n$ in one-to-one correspondence with partitions of $n$ with distinct parts. If we denote the number of partitions of $n$ into distinct parts by $Q(n)$, we have $\log Q(n)\sim \pi\sqrt{n/3}$ \cite{Andrews} and the lower bound in Theorem \ref{thm: main} follows.\\

\begin{prop}~\label{Inequivalent by Partitions}
There exists a family of dynamically inequivalent linear transformations $A:(\F_q)^n\to (\F_q)^n$ in bijection with the partitions of $n$ with distinct parts.
\end{prop}

\begin{proof}
Let $\lambda=(\lambda_1,\dots,\lambda_r)$ be a partition of $n$ with distinct parts. For $1\leq i\leq r$, let $A_i:(\F_q)^{\lambda_i}\to(\F_q)^{\lambda_i}$ be a linear transformation given by the companion matrix of any primitive polynomial of degree $i$ over $\F_q$, and let $A_{\lambda}:(\F_q)^n\to (\F_q)^n$ be defined as $A=A_1\oplus\dots\oplus A_r$. The proposition follows once we show that if $\lambda$ and $\mu$ are unequal partitions of $n$ with distinct parts, then $A_{\lambda}$ and $A_{\mu}$ are not dynamically equivalent.

Let $\lambda=(\lambda_1,\dots,\lambda_r)$ and $\mu=(\mu_1,\dots,\mu_r)$ be unequal partitions of $n$ with distinct parts. There exists some $m\in\mu$ such that $m\notin\lambda$. Any primitive polynomial of degree $i$ has order $q^i-1$, and it follows that the graph $G_{((\F_q)^n,B)}$ contains a cycle of length $q^m-1$ \cite{LN}. We show that $G_{((\F_q)^n,A)}$ has no cycle of this length.

The graph $G_{((\F_q)^{\lambda_i},A_i)}$ consists of one fixed point (the zero vector) and one cycle of length $q^{\lambda_i}-1$ \cite{Elspas}. The cycle lengths in $G_{((\F_q)^n,A)}$ are equal to least common multiples of the cycle lengths of the $G_{((\F_q)^{\lambda_i},A_i)}$ \cite{Elspas}. Therefore every cycle of $G_{((\F_q)^n,A)}$ has length equal to 
\begin{equation*}
\text{LCM}_{i\in S}[q^{\lambda_i}-1]
\end{equation*}
 for some subset $S\subseteq\{1,\dots,r\}$. We need to show that this LCM cannot equal $q^m-1$.

Assume by way of contradiction that $\text{LCM}_{i\in S}[q^{\lambda_i}-1]=q^m-1$. Each $q^{\lambda_i}-1$ divides $q^m-1$, so by the Euclidean Algorithm, each $\lambda_i$ divides $m$. In particular, $m\geq\lambda_i$ for each $i\in S$, but we know $m\neq\lambda_i$ for any $i$, so $m>\lambda_i$ for each $i\in S$. By Zsigmondy's Theorem \cite{Zsigmondy}, with the exception of $(q,m)=(2,6)$, there exists a prime divisor of $q^m-1$ that does not divide $q^k-1$ for $k<m$. This prime cannot divide $\text{LCM}_{i\in S}[q^{\lambda_i}-1]$, which is a contradiction. In the exceptional case $(q,m)=(2,6)$ we can verify directly that $2^6-1=63$ is not the LCM of any collection of smaller numbers of the form $2^{\lambda_i}-1$.
\end{proof}

To prove the upper bound in Theorem \ref{thm: main} we develop a series of preliminary results. Proposition \ref{prop:affine conjugacy} decomposes an affine-linear transformation into the direct sum of an arbitrary linear transformation and an affine transformation whose dynamical equivalence class is determined by an integer partition. Proposition \ref{prop:dynamical equivalence} gives a sufficient condition for dynamical equivalence of linear transformations that depends on the factorization of their characteristic polynomials. Counting characteristic polynomials yields the result.

\begin{lem}\label{lem: ConjugateByParts}
Let $V$ and $W$ be vector spaces, and let the pairs of linear transformations $A,B:V\to V$ and $C,D:W\to W$ be such that $A\sim B$ and $C\sim D$. Then $A\oplus C\sim B\oplus D$ as maps from $V\oplus W$ to itself.
\end{lem}

\begin{proof} We have $A=\phi^{-1}B\phi$ and $C=\psi^{-1}D\psi$ for bijections $\phi:V\to V$ and $\psi:W\to W$. Define $\theta:V\oplus W\to V\oplus W$ as $\phi$ on $V$ and $\psi$ on $W$, extending linearly to $V\oplus W$. Then $A\oplus C=\theta^{-1}(B\oplus D)\theta$.
\end{proof}

\begin{prop}\label{prop:affine conjugacy}
Let $V=(\F_q)^n$ and let $T:V\to V$ be defined as $Tx=Ax+b$, where $A:V\to V$ is linear and $b\in V$. There exists a direct sum decomposition $V=V_1\oplus V_2$ (where it is possible that $V_1=0$) with $T=T_1\oplus T_2$, $T_1x=A_1x+b_1$ and $T_2x=A_2x+b_2$, such that\newline

\begin{enumerate}
\item $T_2\sim A_2$.\\
\item The dynamical equivalence class of $T_1$ is determined by a partition of $\dim V_1$.
\end{enumerate}
\end{prop}

\begin{proof}
This follows from the work of Wang in \cite{Wang}. Wang proves that if there exists an $s\in V$ such that $Ts=s$, then $T\sim A$, which is the $V_1=0$ case of the proposition. If no such $s$ exists, then there exists a decomposition $T=T_1\oplus T_2$ as in the statement of the proposition such that $T_2\sim A_2$. In this case there exists another direct sum decomposition $V_1=W_1\oplus\cdots\oplus W_r$, with $T_1=S_1\oplus\cdots\oplus S_r$, $S_i:W_i\to W_i$ affine-linear, such that if we let $t_i=\dim W_i$, each graph $G_{(W_i,S_i)}$ consists of $\frac{q^{t_i}}{\ord (x-1)^{t_i+1}}$ cycles of length $\ord (x-1)^{t_i+1}$. Therefore the list $(t_1,\dots,t_r)$, which partitions $\dim V_1$, determines the dynamical equivalence class of each $S_i$. This determines the dynamical equivalence class of $T_1$ by Lemma \ref{lem: ConjugateByParts}.
\end{proof}

\begin{lem}\label{lem: ConjugateByOrders}
Let $V=(\F_q)^n$. Let the linear map $A:V\to V$ have characteristic polynomial $f^r$, where $f$ is irreducible over $\F_q$ and $f(0)\neq 0$. The dynamical equivalence class of $A$ is determined by $\deg f$, $\ord f$, and a partition of $r$.
\end{lem}
\begin{proof}
There exist direct sum decompositions $V=V_1\oplus\cdots\oplus V_m$ and $A=A_1\oplus\cdots\oplus A_m$ such that in some basis of each $V_i$, $A_i:V_i\to V_i$ can be written as the companion matrix of $f^{\lambda_i}$ for some $\lambda_i$, and $\sum \lambda_i=r$.

Each graph $G_{(V_i,A_i)}$ is explicitly determined by the data $\ord f$, $\deg f$, and $\lambda_i$ \cite{Elspas}. This determines the dynamical equivalence class of each $A_i$, and by Lemma \ref{lem: ConjugateByParts}, the dynamical equivalence class of $A$.\end{proof}

\begin{rem}
The work of Elspas in \cite{Elspas} is intimately connected with the theory of linearly recurrent sequences over finite fields. In Lemma \ref{lem: ConjugateByOrders}, each $G_{(V_i,A_i)}$ consists of cycles that correspond to all linearly recurrent sequences over $\F_q$ with characteristic polynomial $f^{\lambda_i}$, and the lengths of the cycles are the periods of the sequences. The periods of the sequences and the number of sequences with each period can be computed from the data $\ord f$, $\deg f$, and $\lambda_i$ \cite[Theorem 6.63]{LN}.
\end{rem}

\begin{prop}\label{prop:dynamical equivalence}
Let $V=(\F_q)^n$ and let $A:V\to V$ be a linear transformation. Let $p\in\F_q[x]$ be the characteristic polynomial of $A$ and write its factorization into irreducibles as
\begin{equation*}
p=x^{r_0}\prod_{i=1}^mp_i^{r_i}
\end{equation*}
where the $p_i$ are distinct and no $p_i$ equals $x$. The dynamical equivalence class of $A$ is completely determined by an integer partition of each $r_i$ and two lists of $m$ positive integers: $\{\deg p_i\}$ and $\{\ord p_i\}$.
\end{prop}

\begin{proof} By the theory of the Jordan canonical form \cite{HK} there exist direct sum decompositions $A=A_0\oplus\cdots\oplus A_m$ and $V=V_0\oplus\cdots\oplus V_m$ where $A_i:V_i\to V_i$, the characteristic polynomial of $A_0$ is $x^{r_0}$, and the characteristic polynomial of $A_i$ is $p_i^{r_i}$ for $i\geq 1$. The Jordan form of the nilpotent map $A_0$ is specified by a partition of $r_0$ in which each part is the size of a Jordan block, so this partition determines the similarity class of $A_0$ and hence the dynamical equivalence class (if two linear maps are similar, they are dynamically equivalent).

Suppose that $\deg p_i$ and $\ord p_i$ are given for $i\geq 1$. Specifying a partition of each $r_i$ determines the dynamical equivalence class of each $A_i$ by Lemma \ref{lem: ConjugateByOrders}, which in turn determines the dynamical equivalence class of $A$ by Lemma \ref{lem: ConjugateByParts}.
\end{proof}

Before proceeding with the proof of the upper bound in Theorem \ref{thm: main}, we record a proposition that will be needed at a key moment in the counting argument.

\begin{prop}\label{prop: Maximum Number of Irreducible Factors}
Let $i_q(n)$ denote the maximum possible number of distinct irreducible factors of a degree $n$ polynomial over $\F_q$. Then
\begin{equation*}
i_q(n) = O\left(\frac{n}{\log n}\right).
\end{equation*}
\end{prop}
\begin{proof}
For $q\geq 3$ it is proved in \cite[Lemma A1]{KB} that
\begin{equation*}
i_q(n)\leq\frac{n}{\log_q(n)-3},
\end{equation*}
which immediately implies the proposition when $q\neq 2$. As we only require a weaker big-$O$ estimate, we present a simplified version of the proof in \cite{KB} which also works for $q=2$.

For $f\in\F_q[x]$, let $\omega(f)$ denote the number of distinct irreducible factors of $f$. We construct $f$ such that $\omega(f)=i_q(n)$ by a greedy algorithm. That is, first multiply together all degree 1 irreducibles, then all degree 2 irreducibles, and so on, until multiplying $f$ by another irreducible would raise its degree higher than $n$. Then $\deg f\leq n$ and no degree $n$ polynomials have more distinct irreducible factors than $f$. It suffices to prove the proposition for polynomials of the form
\begin{equation*}
f = g\left(\prod_{\stackrel{\text{irreducible }p\in\F_q[x]}{\deg p <k}} p\right)
\end{equation*}
where $g$ is a product of $m$ irreducible polynomials of degree $k$, each of which appears with multiplicity 1. Let $N_q(j)$ denote the number of irreducible polynomials over $\F_q$ of degree $j$. We have
\begin{equation}\label{eqn: omega}
\omega(f)=\sum_{j=1}^{k-1}N_q(j) + m
\end{equation}
and
\begin{equation}\label{eqn: degree}
\deg(f) = \sum_{j=1}^{k-1}jN_q(j) + mk\leq n.
\end{equation}

We now show that the inequality
\begin{equation}\label{eqn: N_qInequality}
\sum_{j=1}^{k}N_q(j)\leq \frac{3}{k}\sum_{j=1}^k jN_q(j)
\end{equation}
holds for large $k$. As $N_q(j)\leq q^j/j$ \cite{LN} we have
\begin{align*}
\sum_{j=1}^k N_q(j) \leq \sum_{j=1}^k \frac{q^j}{j} & \leq 
\sum_{1\leq j\leq k/2} q^j + \sum_{k/2<j\leq k} \frac{q^j}{k/2}\\
& \leq \frac{q^{k/2+1}-q}{q-1} + \frac{2q^{k+1}-2q^{k/2}}{k(q-1)}\\
& \leq \frac{1}{q-1}\left(q^{k/2+1}+\frac{2q^{k+1}}{k}\right).
\end{align*}
As $N_q(j)\geq q^j/j - q^{j/2+1}/(j(q-1))$ \cite{LN} we have
\begin{align*}
\frac{3}{k}\sum_{j=1}^k jN_q(j) & \geq\frac{3}{k}\left(\sum_{j=1}^k q^j-\frac{q^{j/2+1}}{q-1}\right) \\
& = \frac{1}{q-1}\left(\frac{3q^{k+1}-3q}{k} - \frac{3q^{(k+1)/2+1}-3q^{3/2}}{k(q-1)}\right).
\end{align*}
Therefore equation \ref{eqn: N_qInequality} holds if
\begin{equation*}
q^{k/2+1}+\frac{2q^{k+1}}{k} \leq \frac{3q^{k+1}-3q}{k} - \frac{3q^{(k+1)/2+1}-3q^{3/2}}{k(q-1)}.
\end{equation*}
or equivalently
\begin{equation*}
q^{k/2+1}+\frac{3q}{k} + \frac{3q^{(k+1)/2+1}}{k(q-1)} \leq \frac{q^{k+1}}{k} + \frac{3q^{3/2}}{k(q-1)}.
\end{equation*}
Comparing powers of $q$ on both sides, it is clear that this inequality holds for large $k$.

Returning to equations \ref{eqn: omega} and \ref{eqn: degree}, we use \ref{eqn: N_qInequality} to conclude
\begin{equation*}
i_q(n)=\omega(f) = \sum_{j=1}^{k-1}N_q(j) + m\leq \frac{3}{k}\left(\sum_{j=1}^{k-1}jN_q(j)+mk\right)\leq \frac{3n}{k}.
\end{equation*}
It only remains to show $k\geq C\log n$ for some $C$. By our construction of $f$, the largest that $n$ can be for a given $k$ occurs when $g$ is the product of all degree $k$ irreducible polynomials over $\F_q$. For this $g$, $n\leq k+\sum_{j=1}^k jI_j$ because if $n$ exceeded this amount, we could add a degree $k+1$ irreducible factor to $f$. So
\begin{equation*}
n\leq k+\sum_{j=1}^k jI_j\leq k+\sum_{j=1}^k q^k\leq k+\frac{q^{k+1}-q}{q-1} \leq k+q^{k+1}.
\end{equation*}
When $k$ is large, $k<q^{k+1}$. Recall that $q\geq 2$. These imply that $n\leq k+q^{k+1}\leq q^{k+2}$, so $\log_q(n)\leq k+2$. For $n\geq q^4$
\begin{equation*}
k\geq \log_q(n)-2\geq \frac{1}{2}\log_q(n),
\end{equation*}
which completes the proof.
\end{proof}

We now combine Propositions \ref{prop:affine conjugacy} and \ref{prop:dynamical equivalence} to give an upper bound on $D_q(n)$, completing the proof of Theorem \ref{thm: main}.

\begin{thm}\label{thm:count of affine dynamics}
\begin{equation*}
D_q(n) = \exp\left(O\left(\frac{n}{\log\log n}\right)\right).
\end{equation*}
\end{thm}

\begin{proof}
Let $V=(\F_q)^n$. Let $T:V\to V$ be defined by $Tx=Ax+b$, where $A$ is linear and $b\in V$. By Proposition \ref{prop:affine conjugacy}, $V=V_1\oplus V_2$ and $T=T_1\oplus T_2$ such that the dynamical equivalence class of $T_1$ is determined by a partition of $\dim V_1$ and $T_2\sim A_2$ for some linear $A_2:V_2\to V_2$. (It may be the case that $V_1=0$ and $T_1=0$.) Let 
\begin{equation*}
p=x^{r_0}\prod_{i=1}^mp_i^{r_i}
\end{equation*}
be the characteristic polynomial of $A_2$. Note that $\deg p\leq n$. By Proposition \ref{prop:dynamical equivalence}, the dynamical equivalence class of $A_2$ is determined by a partition of each $r_i$ and the two lists of $m$ integers $\{\deg p_i\}$ and $\{\ord p_i\}$. We estimate the number of ways of specifying these data. Assume for the moment that the $\deg p_i$ and the partitions of the $r_i$ are given and that we need to assign orders to the $p_i$.

Let $d_i=\deg p_i$. By Proposition \ref{prop:order count}, the number of possible ways to assign the $\ord p_i$ is
\begin{equation*}
\prod_{i=1}^m O_q(d_i)\leq \prod_{i=1}^m\tau(q^{d_i}-1).
\end{equation*}
We split this into two products over the ranges $d_i<d$ and $d_i\geq d$ for some $d$ to be chosen later. First we estimate the quantity
\begin{equation*}
C_1=\prod_{d_i<d}\tau(q^{d_i}-1).
\end{equation*}
Using the trivial estimate $\tau(x)< x+1$,
\begin{equation*}
C_1 < \prod_{d_i<d}q^{d_i} = q^{\sum_{d_i<d} d_i}.
\end{equation*}
Each $d_i$ is the degree of a distinct irreducible polynomial over $\F_q$. As in the proof of Proposition \ref{prop: Maximum Number of Irreducible Factors}, $N_q(k)\leq\frac{q^k}{k}$, so
\begin{equation*}
\sum_{d_i<d}d_i\leq \sum_{k=1}^{d-1} kI_k \leq \sum_{k=1}^{d-1} q^k \leq q^d. 
\end{equation*}
Therefore $C_1< q^{q^d}$.

Now we estimate
\begin{equation*}
C_2=\prod_{d_i\geq d}\tau(q^{d_i}-1).
\end{equation*}
By the estimate on $\tau(x)$ in \cite[Theorem 8.8.9]{ANT}, there exists $c$ such that
\begin{equation*}
\tau(q^{d_i}-1)\leq 2^\frac{c\log(q^{d_i}-1)}{\log\log(q^{d_i}-1)}.
\end{equation*}
This implies
\begin{equation*}
C_2\leq\prod_{d_i\geq d}2^\frac{cd_i\log q}{\log\log(q^d-1)}\leq
2^\frac{c\log q\sum_{d_i\geq d} d_i}{\log\log q^{d-1}}
\end{equation*}
where we use the inequality $q^d-1\geq q^{d-1}$, which is true for $q\geq 2$ and $d\geq 1$. Also, $\sum_{d_i\geq d} d_i\leq n$, so
\begin{equation*}
C_2\leq 2^\frac{c(\log q)n}{\log (d-1) + \log\log q}.
\end{equation*} 

Putting these estimates together we have
\begin{equation*}
\prod_{i=1}^m\tau(q^{d_i}-1) = C_1C_2\leq \exp\left(q^d\log q + \log 2\frac{c(\log q)n}{\log (d-1) + \log\log q} \right).
\end{equation*}
Choose $d=\frac{\log n}{2\log q}$ and note $\log(d-1)\geq \log(d/2)$ for $d\geq 1$. Then
\begin{align*}
q^d\log q + \log 2\frac{c(\log q)n}{\log (d-1) + \log\log q}  & \leq
n^{1/2}\log q + \frac{cn\log 2\log q}{\log \frac{\log n}{4\log q} + \log\log q}\\
& = O\left(\frac{n}{\log\log n}\right)
\end{align*}

Now we estimate the number of ways that the $r_i$ and $\deg p_i$ can occur. Because $\sum_{i=1}^m r_i\deg p_i=\dim V_2$, the $\deg p_i$ form a partition of $\dim V_2$ in which each appears $r_i$ times. This is a ``factorization pattern" of $\dim V_2$ as in \cite{Factorization} which is specified by first picking a partition of $\dim V_2$ into parts of size $k$, each of which occurs $s_k$ times, and then further dividing each $s_k$ into parts $r_i$. Let $b(n)$ denote the number of factorization patterns of $n$. It is mentioned in \cite{Partitions} and proved in \cite{Mitchell} that
\begin{equation*}
b(n) = \exp\left(B\sqrt{n\log n}+O(\sqrt{n})\right).
\end{equation*}

Finally, we need to choose a partition of each $r_i$ and a dynamical equivalence class for $T_1$ given by a partition of $\dim V_1$. If $P(x)$ denotes the partition function, we have $P(x)\leq\exp(K\sqrt{x})$ \cite{Andrews}. The number of ways to specify all these partitions is
\begin{align*}
&P(\dim V_1)\prod_{i=0}^m P(r_i) \leq P(n)\exp\left(K\sum_{i=0}^m \sqrt{r_i}\right)\\
& \leq \exp\left(K\sqrt{n}+K\sqrt{m+1}\sqrt{\sum_{i=0}^{m} r_i}\right)
\leq \exp(K\sqrt{n}(1+\sqrt{m+1}))\\
& = \exp\left(K\sqrt{n}\left(1+O\left(\sqrt{\frac{n}{\log n}}\right)\right)\right)
= \exp\left(O\left(\frac{n}{\sqrt{\log n}}\right)\right).
\end{align*}
The inequality $\sum_{i=0}^m \sqrt{r_i}\leq \sqrt{m+1}\sqrt{\sum_{i=0}^{m} r_i}$ follows from the standard fact that the arithmetic mean of the $\sqrt{r_i}$ is at most the root mean square, and $m+1$ is the number of distinct irreducible factors of $p$, so $m+1 = O(n/\log{n})$ by Proposition \ref{prop: Maximum Number of Irreducible Factors}.

Putting this all together, the number of ways to choose a dynamical equivalence classes for $T_1$ and $T_2$, and therefore a dynamical equivalence class for $T$ by Lemma \ref{lem: ConjugateByParts}, is at most
\begin{equation*}
b(n)\exp\left(O\left(\frac{n}{\sqrt{\log n}}\right)\right)\exp\left(O\left(\frac{n}{\log\log n}\right)\right) = \exp\left(O\left(\frac{n}{\log\log n}\right)\right) 
\end{equation*}
which completes the proof of Theorem \ref{thm: main}
\end{proof}

\begin{rem}
It seems possible that the estimates in Theorem \ref{thm:count of affine dynamics} could be improved. The main estimate used for $\tau(q^d-1)$ is the worst-case estimate on $\tau(x)$ that follows from the prime number theorem. It may be possible to give a better estimate based on the distribution of multiplicative orders of $q$ modulo various integers $n$. (If $n$ divides $q^d-1$, then $d$ is a multiple of the multiplicative order of $q$ mod $n$.)  Questions along these lines tend to be difficult, even for $q=2$. See \cite{Arnold}, \cite{KP}, \cite{AverageOrders}, \cite{MRS}, and \cite{Pomerance} for some related work.
\end{rem}

\subsection*{Acknowledgements}
This research was partly supported by NSF grants no. CCF-0635355 and EMSW21-RTG. The second author would like to thank ICERM for an invitation to the Spring 2012 Semester Program on Complex and Arithmetic Dynamics, which facilitated the development of many of these ideas. We would like to thank Joseph Silverman for helpful comments and observations on a preliminary version of this paper, and Michael Zieve for suggesting the connection to group theory.

\bibliographystyle{plain}
\bibliography{AffineLinearFunctionalGraphs}
\nocite{*}
\end{document}